\documentclass{amsart}

\usepackage{a4wide}
\usepackage{amsmath,amssymb,amsthm}
\usepackage{tikz}
\usepackage{graphicx}
\usepackage{hyperref}
\hypersetup{colorlinks=true,linkcolor=black,citecolor=blue}
\usepackage[capitalize]{cleveref}

\title[A note on DH $\cap$ co-DH]{A note on distance-hereditary graphs whose complement is also distance-hereditary}
\author{Hugo Jacob}
\thanks{LIRMM, Université de Montpellier, CNRS, France.}
\thanks{Research supported by ANR project GODASse ANR-24-CE48-4377.}
\date{}

\def\DHcoDH{DH $\cap$ co-DH}

\theoremstyle{definition}
\newtheorem{lemma}{Lemma}
\newtheorem{theorem}{Theorem}
\newtheorem{observation}{Observation}
\newtheorem{corollary}{Corollary}

\begin{document}

\begin{abstract}
Distance-hereditary graphs are known to be the graphs that are totally
decomposable for the split decomposition. We characterise distance-hereditary
graphs whose complement is also distance-hereditary by their split
decomposition and by their modular decomposition.
\end{abstract}
\maketitle

\vspace{10pt}

Distance-hereditary graphs\footnote{We will use common graph theoretic
notations mostly coherent with Diestel's textbook \cite{DiestelTextbook}.
$N[x]$ denotes $N(x) \cup \{x\}$, the \emph{closed neighbourhood} of $x,$ and
$P_k$ is the path of order $k$.}
constitute a well-known hereditary\footnote{A class of graphs is
\emph{hereditary} if it is closed under induced subgraphs.} class of
graphs which admits many characterisations.
Their name comes from their characterisation as graphs whose connected induced
subgraphs preserve distances. They are also characterised by a simple list of
forbidden induced subgraphs: holes, the house, the domino, and the gem (see
\cref{fig:HHDG}) \cite{BANDELT86}.
As such, distance-hereditary graphs are also called HHDG-free
graphs. They also admit a characterisation via a sequence of eliminations of
twins and pendant vertices \cite{BANDELT86}. Two vertices $u,v$ are
\emph{twins} if $N(u)\setminus \{u,v\} = N(v)\setminus \{u,v\}$, a
\emph{pendant} vertex is a vertex of degree $1$. Finally, a graph is
distance-hereditary if it is totally decomposable by the split decomposition.
This last characterisation will be explained in further detail later.

\begin{figure}[h]
\centering
\includegraphics[scale=0.8]{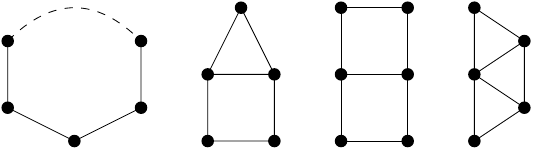}
\caption{The forbidden induced subgraphs corresponding to the class of
distance-hereditary. From left to right: holes (cycles of length at least $5$),
house, domino, and gem.}
\label{fig:HHDG}
\end{figure}

The split decomposition was first introduced by Cunningham \cite{SplitDecIntro}.
We briefly introduce some notation, see e.g. \cite{GioanP12} for a more
detailed presentation.
Given a connected graph $G=(V,E)$, a \emph{split} is a
bipartition $(A,B)$ of $V$ such that $|A|,|B| \geq 2$ and such that there are
all edges between $N(A)$ and $N(B)$. A split $(A,B)$ is called \emph{strong} if
it does not overlap another split $(A',B')$, i.e. one of $A \cap A'$, $B \cap
A'$, $A \cap B'$, $B \cap B'$ must be empty. Decomposing using strong splits
produces a unique decomposition, which we call the (canonical) split
decomposition. The \emph{split decomposition} $(T,\mu)$ of a graph $G$ can be
described via an \textbf{undirected} tree $T$ whose leaves are bijectively
mapped to $V(G)$ (see examples in \cref{fig:P4Bull}).
Each internal node $t$ of $T$ is labelled by an induced
subgraph $\mu(t)$ of $G$ corresponding to picking\footnote{The choice is not
arbitrary. A suitable vertex may be found by following a path which alternates
between edges of $T$ and edges of graphs labelling nodes of $T$, see also
\cref{lem:adj-split-dec}. For instance, consider the split decomposition of the
bull illustrated in \cref{fig:P4Bull}. Let $t$ be the internal node adjacent to
the leaf mapped to vertex $e$. The interface vertices of $\mu(t)$ may be
replaced by vertices $b$ and $c$ but not vertices $a$ and $d$.},
in each subtree incident to $t$, one vertex of $V(G)$ mapped to a leaf.
The vertices of $\mu(t)$ are thus in bijection with edges incident to $t$ in
$T$. We call \emph{interface vertex} a vertex of some graph $\mu(t)$ mapped
to an edge of $T$ not incident to a leaf. 
The strong splits of $G$ correspond exactly to the bipartitions $(A,B)$ of the
set of leaves of $T$ for which $A$ and $B$ are the respective sets of leaves in 
the two components of $T-e$, for some internal edge $e$ of $T$.
The important property of the split decomposition is the following.

\begin{lemma}[\cite{GioanP12}]\label{lem:adj-split-dec}
	Two vertices $u,v \in V(G)$ are adjacent in $G$ if and only if, for every
	internal node $t$ of $T$ on the path $P$ between the leaves mapped to $u$ and
	$v$, there is an edge in $\mu(t)$ between the vertices of $\mu(t)$ mapped
	to the edges of $P$.
\end{lemma}

The following known observation will be implicitly used to consider the
possible structure of split decompositions.

\begin{observation}[see `reduced' in \cite{GioanP12}]
	Internal nodes of a split decomposition have degree at least $3$.
\end{observation}

The fact that a distance-hereditary graph is totally decomposable by the split
decomposition can be expressed more explicitly as follows.

\begin{theorem}[\cite{GioanP12}]\label{thm:totally-decomposable}
	A connected graph $G$ is distance-hereditary if and only if its split decomposition
	$(T,\mu)$ has all its internal nodes $t$ labelled by graphs $\mu(t)$ that are stars or
	cliques.
\end{theorem}

Some other hereditary classes related to distance-hereditary graphs can be
characterised by the structure of their split decomposition.
\emph{Cographs} are exactly $P_4$-free graphs.
An \emph{asteroidal triple} (AT) is a triple of vertices such that each pair of
vertices is connected by a path avoiding the closed neighbourhood of the third
vertex. A graph is \emph{AT-free} if it contains no asteroidal triple.

\begin{theorem}[\cite{GioanP12}]\label{thm:dec-cograph}
	A connected graph $G$ is a cograph if and only if its split decomposition $(T,\mu)$
	has all its internal nodes $t$ labelled by graphs $\mu(t)$ that are stars
	or cliques, and there is an edge of $T$ towards which all stars are
	pointing\footnote{A graph $\mu(t)$ that is a star points in the direction
	of the edge $e$ incident to $t$ to which the centre of the star is mapped.
	It points towards some node or edge of $T$ if the path leading to it from
	$t$ contains $e$.}.
\end{theorem}

\begin{theorem}[\cite{twwOne}]
	A connected graph $G$ is distance-hereditary and AT-free
	if and only if its split decomposition $(T,\mu)$ has all its internal nodes
	$t$ labelled by graphs $\mu(t)$ that are stars or cliques and there is a
	path $P$ of $T$ towards which all stars that label nodes $t \in V(T)\setminus
	P$ are pointing.
\end{theorem}

We now briefly introduce \emph{modules}.
A \emph{module} in a graph
$G=(V,E)$ is a subset $M$ of $V$ such that for every vertex $v \in V \setminus
M,$ $N(v) \cap M \in \{M,\varnothing\}$. A module is trivial if it consists of
the empty set, a
single vertex or the entire vertex set $V$. In particular, connected components
and pairs of twins are modules. A graph is \emph{prime}\footnote{While this word 
can also be used for graphs that do not admit splits, we will restrict its use to
graphs that admit only trivial modules.} (with respect to modules) if
it has only trivial modules. A module $M$ is \emph{strong} if it does not
overlap other modules: for every other module $M'$ we have $M \cap M' =
\varnothing$, $M \subseteq M'$ or $M' \subseteq M$. The \emph{modular
decomposition} of a graph $G$ is a \textbf{rooted} decomposition tree
representing the family of strong modules of $G$.
Vertices of $G$ are mapped to the leaves, and the sets of leaves of (rooted)
subtrees are exactly the strong modules. The internal nodes of this tree are
labelled by a (quotient) graph which encodes the adjacency between the strong
modules corresponding to its incident subtrees. It is well-known that cographs
are exactly graphs whose modular decomposition is labelled only by cliques and
independent sets (this corresponds exactly to the cotree). \cref{thm:dec-cograph}
simply expresses that it can be replaced by a split decomposition by replacing
independent sets with stars pointing to the root.

We will now consider graphs that are distance-hereditary and whose
complement is also distance-hereditary. It is direct from the forbidden induced
subgraphs of the class of distance-hereditary graphs to obtain the forbidden
induced subgraphs of its complementation closed subclass (see
\cref{fig:DHcoDH}). We denote this subclass by \DHcoDH{} (the class of graphs $G$
that are distance-hereditary and whose complement $\overline{G}$ is also
distance-hereditary).

\begin{figure}[h]
\centering
\includegraphics[scale=0.9]{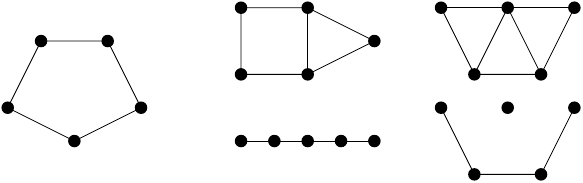}
\caption{The forbidden induced subgraphs of \DHcoDH{}. From left to right:
$C_5$, the house and its complement $P_5$, the gem and its complement $K_1 +
P_4$. Note that larger holes and the domino contain $P_5$ as an induced subgraph.
}
\label{fig:DHcoDH}
\end{figure}

\begin{figure}[h]
\centering
\includegraphics[scale=0.9]{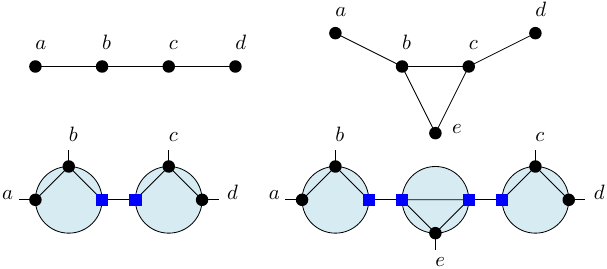}
\caption{From left to right: $P_4$ and the bull, the only two prime graphs in
\DHcoDH{} above their respective split decompositions. The graphs $\mu(t)$ are
depicted inside light blue circles, and the edges of $T$ are outside these circles.
Interface vertices are depicted by blue squares, original vertices of the
decomposed graph are depicted by black discs.
}
\label{fig:P4Bull}
\end{figure}

We obtain a characterisation of prime graphs of \DHcoDH{} and deduce a
characterisation of graphs of \DHcoDH{} by the structure of their split
decomposition or modular decomposition. The theorems will be proved after
providing further observations.

\begin{theorem}\label{thm:prime-DHcoDH}
	 The prime graphs in \DHcoDH{} are $P_4$ and the bull (see \cref{fig:P4Bull}).
\end{theorem}

\begin{theorem}\label{thm:dec-DHcoDH}
	A connected graph $G$ is in \DHcoDH{} if and only if the following
	conditions are satisfied. Its split decomposition $(T,\mu)$ has all its
	internal nodes $t$ labelled by graphs $\mu(t)$ that are stars or
	cliques. There exists either an edge $e,$ or two nodes $s,s',$ labelled by
	stars, each not pointing towards the other, which are either adjacent in $T$ or
	have exactly one node $s''$ labelled by a clique on the path between them.
	All nodes of $T$ labelled by stars except $s,s'$ point towards $e$ or the
	nodes $s,s',s''$.
\end{theorem}

We provide some explanation of the different cases in the statement above. The
case of the edge $e$ is exactly the case of a cograph (\cref{thm:dec-cograph}),
while the two other cases correspond to containing either $P_4$ or the bull as
an induced subgraph, see \cref{fig:P4Bull} for their split decompositions.

\begin{lemma}\label{lem:subcographs}
If $G$ is in \DHcoDH{}, for every vertex $v$, $G[N(v)]$ and $G-N[v]$ are cographs.
\end{lemma}

\begin{proof}
Since $v$ cannot be the centre of a gem or a co-gem, $G[N(v)]$ and $G-N[v]$ are
$P_4$-free, which is a well-known characterisation of cographs.
\end{proof}

A \emph{cutvertex} is a vertex whose deletion increases the number of
connected components. A \emph{biconnected component} is a maximal edge-induced
subgraph which cannot be disconnected by deleting a single vertex. By
definition, a cutvertex is always contained in at least two biconnected
components. More generally, the behaviour of cutvertices and biconnected
components is described by block-cut trees. If $C$ is a biconnected component
and $v$ is a cutvertex contained in $C$, then the edges of $C$ incident to $v$
form an edge-cut which corresponds to a strong split (unless this edge-cut is a
single edge incident to a pendant vertex). This can be expressed in terms of
the structure of the split decomposition as follows.

\begin{observation}
	A cutvertex is exactly a vertex mapped to a leaf $\ell$ of the split
	decomposition $(T,\mu)$ adjacent to an internal node $t$ labelled by a
	graph $\mu(t)$ which is a star whose centre is mapped to the edge of $T$
	incident to $\ell$. Furthermore, a graph $\mu(t)$ which is not a star must
	be biconnected.
\end{observation}

\begin{lemma}\label{lem:cutvertices}
	If $G$ contains $3$ cutvertices in the same connected component, then $G$
	is not in \DHcoDH{}.
\end{lemma}

\begin{proof}
	Consider $3$ cutvertices in a connected graph $G$. From the above
	observation, they correspond to three nodes $t_1,t_2,t_3$ in the split
	decomposition such that $\mu(t_1),\mu(t_2),\mu(t_3)$ are stars. Without
	loss of generality, we may ask that on paths of the split decomposition
	tree between $t_1,t_2,t_3,$ each node $t$ labelled by a star points to one
	of $t_1,t_2,t_3$. Indeed, otherwise we can consider the induced subgraph
	obtained from replacing the subtree pointed at by the star's centre by a
	single vertex (if necessary), this makes it a cutvertex. We can then
	replace one of $t_1,t_2,t_3$ which is not between the other two by $t$,
	this decreases the number of nodes on paths between $t_1,t_2,t_3,$
	guaranteeing that we reach the claimed case after a finite number of steps.
	There are two possible configurations for $t_1,t_2,t_3$:
	either they are on a common path of $T,$
	or there is a distinct branching node which separates the three nodes, see
	\cref{fig:cutvertices}.

	In the first case, we can consider an induced subgraph of $G$ consisting of
	the three cutvertices and a private neighbour for each cutvertex except the
	central one.
	Indeed, each path of the split decomposition between the cutvertices is
	labelled by cliques or stars pointing to either cutvertex. We can then
	conclude that these vertices induce a $P_5$ using \cref{lem:adj-split-dec},
	which contradicts membership in \DHcoDH{}, see \cref{fig:DHcoDH}.

	In the second case, we can assume the branching node is labelled by a
	clique since otherwise $G$ would contain an induced subgraph in the previous
	configuration. We can then conclude that $K_1+P_4$ (the co-gem) is an
	induced subgraph, which also contradicts membership in \DHcoDH{}, see \cref{fig:DHcoDH}.
\end{proof}

\begin{figure}[h]
\centering
\includegraphics[scale=0.9]{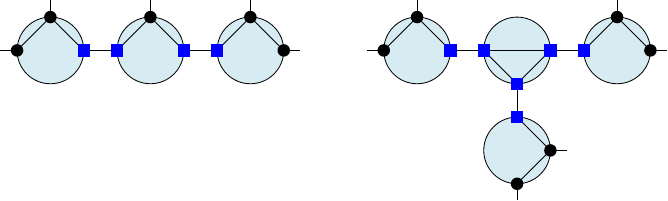}
\caption{An illustration of the two configurations in the proof of \cref{lem:cutvertices}.}
\label{fig:cutvertices}
\end{figure}

\begin{proof}[Proof of \cref{thm:prime-DHcoDH}]
	First, $P_4$ and the bull are isomorphic to their complement, thus it
	suffices to check that they are distance-hereditary. One can check that they
	do not contain any of the obstructions of \cref{fig:HHDG} as induced
	subgraphs. Moreover, both graphs are prime (this is well-known and easy to check,
	note that using \cref{lem:twin-free-connected-DH} makes it easier).

	Now, assume that $G$ is prime and in \DHcoDH{}. In particular, it is connected (by being
	prime), and admits a split decomposition $(T,\mu)$ labelled only by stars
	and cliques (by \cref{thm:totally-decomposable} from being connected and in DH).
	Since $G$ is prime, for every internal node $t$ of $T$ labelled by
	a graph $\mu(t)$ having a single interface vertex, $\mu(t)$ is a $P_3$ (the
	star on three vertices) whose centre is not the interface vertex. Otherwise,
	there would be twins in $G$, and thus $G$ would not be prime. In
	particular, the vertex introduced at the leaf incident to the edge of $T$
	mapped to the centre of the $P_3$ is a cutvertex.

	By \cref{lem:cutvertices}, $G$ contains at most $2$ cutvertices. We deduce
	that the split decomposition of $G$ has only two internal nodes having a
	single interface vertex. Note that if there is only a single node in the
	split decomposition, then $G$ is a cograph, which contradicts the fact that
	it is prime. Otherwise, there are at least two nodes having a single
	interface vertex. In particular, we can conclude that there must be exactly
	$2$ cutvertices in $G$. This implies that the internal nodes of $(T,\mu)$
	appear on a path. By \cref{lem:subcographs}, we deduce that the induced
	subgraph $H$ obtained by deleting one of the cutvertices is a cograph.
	Indeed, one connected component of $H$ is a single vertex $v$ (the pendant
	vertex adjacent to the deleted cutvertex) and the other is the graph $G-N[v]$.
	
	We may now apply \cref{thm:dec-cograph} to $G-N[v]$ to deduce the
	structure of the split decomposition $(T,\mu)$. Its split decomposition
	consists only of cliques and stars pointing towards the edge of $T$
	incident to the leaf mapped to the remaining cutvertex. By applying this
	reasoning to both cutvertices, we conclude that nodes labelled by stars other
	than the nodes corresponding to the two cutvertices must point in opposite
	directions simultaneously, a contradiction. We deduce that all internal
	nodes of the split decomposition are cliques except for the nodes
	corresponding to the cutvertices. 

	Now, if two nodes labelled by cliques are adjacent in a split
	decomposition, the edge between them does not correspond to a strong split.
	Hence, in the canonical split decomposition of $G$, there is at most one
	node labelled by a clique. If it exists, it has exactly two interface
	vertices. This implies that it has exactly one vertex which is not an
	interface vertex (otherwise there would be twins), so it is a triangle. We
	can conclude that the split decomposition is either that of a $P_4$ or that
	of a bull, see \cref{fig:P4Bull}.
\end{proof}

We now extend our result to connected graphs in \DHcoDH{} using the
following observation.

\begin{lemma}\label{lem:twin-free-connected-DH}
	If $G$ is a connected distance-hereditary graph, then it is twin-free if
	and only if it is prime.
\end{lemma}

\begin{proof}
	Assume towards a contradiction that $G$ is twin-free but not prime. Hence,
	there exists a nontrivial module $M$ of $G$.
	Since $G$ is connected, there is a vertex $v$ universal to $M$. Since $G$
	is twin-free and $M$ is a module, $G[M]$ is also twin-free. In a cograph,
	there is always a pair of twins (a well-known property which can be deduced
	from \cref{thm:dec-cograph}), so
	$G[M]$ is not a cograph and, in particular, it contains an induced $P_4$.
	Since $v$ is universal to $M$, it is universal to the induced $P_4$ meaning
	$G$ has an induced subgraph isomorphic to the gem, a contradiction.

	Conversely, if $G$ is prime, it is twin-free because a pair of twins is a
	nontrivial module.
\end{proof}

\begin{proof}[Proof of \cref{thm:dec-DHcoDH}]
\cref{thm:prime-DHcoDH} gives possible prime graphs (see \cref{fig:P4Bull}),
\cref{lem:twin-free-connected-DH} explains that other graphs are only obtained
by adding twins, which corresponds to replacing vertices by cographs, and
\cref{thm:dec-cograph} describes cographs in terms of split decompositions.
\end{proof}

The following is an equivalent statement of \cref{thm:dec-DHcoDH}.

\begin{corollary}
	The connected graphs in \DHcoDH{} are exactly the connected graphs that admit a
	sequence of twin eliminations to an induced subgraph of the bull.
\end{corollary}

The case of disconnected graphs follows from forbidding the co-Gem as an induced subgraph.

\begin{observation}
	If $G$ is a disconnected graph in \DHcoDH{}, each connected component is a
	cograph. Hence, $G$ is a cograph.
\end{observation}

We may finally conclude on a complete characterisation of graphs in \DHcoDH{}.

\begin{theorem}\label{thm:modular-DHcoDH}
	Graphs in \DHcoDH{} are exactly graphs admitting a sequence of twin eliminations to an
	induced subgraph of the bull.

	Equivalently, graphs in \DHcoDH{} are exactly graphs whose modular
	decomposition has its root labelled by an induced subgraph of the bull,
	and other nodes are labelled cliques or independent sets.
\end{theorem}

Since the bull is AT-free and the class of AT-free graphs is stable under the
addition of twins, we can conclude the following.

\begin{corollary}
	\DHcoDH{} is a subclass of the class DH $\cap$ AT-free.
\end{corollary}

We conclude with some consequences on the recognition of graphs in \DHcoDH{}.
Since the split decomposition and the modular decomposition of a graph can be
computed in linear time \cite{SplitAlgo,McConnellS99,CournierH94}, we can either check that
a graph satisfies the conditions of \cref{thm:dec-DHcoDH} or is a cograph, or
check that it satisfies the conditions of \cref{thm:modular-DHcoDH}. Linear
time could already be obtained by checking separately if $G$ and its complement
$\overline{G}$ are distance-hereditary, see \cite{DuboisGM08}.

\begin{corollary}
	Graphs in \DHcoDH{} can be recognised in linear time.
\end{corollary}

\bibliographystyle{alpha}
\bibliography{refs}

\end{document}